\documentclass[11pt]{article}

\usepackage{amsmath, amssymb, amsthm} 
\usepackage{dsfont}
\usepackage{hyperref}

\title{Trimming forests is hard (unless they are made of stars)}

\author{Lior Gishboliner\thanks{Department of Mathematics, ETH, Z\"urich, Switzerland. Email: lior.gishboliner@math.ethz.ch. Research supported in part by SNSF grant 200021\_196965.}
	\and Yevgeny Levanzov\thanks{School of Mathematics, Tel Aviv University, Tel Aviv 69978, Israel. Email: yevgenyl@mail.tau.ac.il.
	}
	\and Asaf Shapira\thanks{
		School of Mathematics, Tel Aviv University, Tel Aviv 69978, Israel.
		Email: asafico@tau.ac.il. Supported in part by ISF Grant 1028/16 and ERC Starting Grant 633509.}
}

\date{}

\allowdisplaybreaks
\expandafter\let\expandafter\oldproof\csname\string\proof\endcsname
\let\oldendproof\endproof
\renewenvironment{proof}[1][\proofname]{%
	\oldproof[\bf #1]%
}{\oldendproof}

\allowdisplaybreaks

\theoremstyle{plain}
\newtheorem{theorem}{Theorem}
\newtheorem{lemma}{Lemma}[section]

\newtheorem{corollary}[theorem]{Corollary}
\newtheorem{conjecture}[lemma]{Conjecture}
\newtheorem{problem}[lemma]{Problem}

\newtheorem{definition}[lemma]{Definition}

\RequirePackage[normalem]{ulem} 
\RequirePackage{color}\definecolor{RED}{rgb}{1,0,0}\definecolor{BLUE}{rgb}{0,0,1} 

\newcommand{\ex}{\text{ex}}

\newcommand{\rem}{\mathrm{Rem}}

\begin{document}
\date{}
\maketitle

\begin{abstract}

Graph modification problems ask for the minimal number of vertex/edge additions/deletions
needed to make a graph satisfy some predetermined property. A (meta) problem of this type, which was
raised by Yannakakis in 1981, asks to determine for which properties
${\mathcal P}$, it is NP-hard to compute the smallest number of edge deletions needed to make a
graph satisfy ${\mathcal P}$. Despite being extensively studied in the past 40 years, this
problem is still wide open. In fact, it is open even when ${\mathcal P}$ is the property of
being $H$-free, for some fixed graph $H$. In this case we use $\rem_{H}(G)$ to
denote the smallest number of edge deletions needed to turn $G$ into an $H$-free graph.

Alon, Sudakov and Shapira [Annals of Math. 2009] proved that if $H$ is not bipartite, then
computing $\rem_{H}(G)$ is NP-hard. They left open the problem of classifying the bipartite graphs
$H$ for which computing $\rem_{H}(G)$ is NP-hard. In this paper we resolve this problem when $H$ is
a forest, showing that computing $\rem_{H}(G)$ is polynomial-time solvable if $H$ is a star forest and NP-hard otherwise.
Our main innovation in this work lies in introducing a new graph theoretic approach for Yannakakis's problem, which differs significantly from all prior works on this subject. In particular, we prove new results concerning an old and famous conjecture of Erd\H{o}s and S\'os, which
are of independent interest.

\end{abstract}


\section{Introduction}\label{sec:intro}

\subsection{Background on graph modification problems}

Graph modification problems are problems of the following nature: we fix a graph property
${\mathcal P}$ and the type of modifications one is allowed to perform on a graph, such as vertex removal
and/or edge removal/addition. Now, given a graph $G$, one would like to compute the minimal number
of operations one needs to perform in order to turn $G$ into a graph satisfying ${\mathcal P}$.
The systematic study of problems of this type was introduced by Yannakakis \cite{Yannakakis0,Yannakakis1}
in the late 70's, and they have been extensively studied ever since. We should point
that besides their intrinsic theoretical importance, graph modification problems also have various practical
applications, see the discussions in \cite{FHS,Mancini}.

Since graph modification problems have been extensively studied in the past four decades,
we will not be able to cover all the relevant background, but rather mention several selected results.
For example, there are numerous works studying graph modification problems of various specific properties such as
planarity \cite{CMT}, being a cograph \cite{LMP} or having a certain degree sequence \cite{CMPPS,G,Mathieson}.
There are also many results dealing with the FPT aspects of various graph modification problems, see, e.g., \cite{Cai,CDFG,FGT,MS} and references therein. More relevant to our investigation here are results trying to obtain statements regarding general families of properties.
One notable result of this type was obtained by Lewis and Yannakakis \cite{LY}, who proved that for every
(non-trivial) {\em hereditary}\footnote{A graph property is hereditary if it is closed under vertex removal.} property, it is NP-hard to compute the smallest number of vertices that need
to be removed to make $G$ satisfy~${\mathcal P}$.

In this paper we focus on edge-deletion problems, and denote by $\rem_{\mathcal P}(G)$
the smallest number of edges whose removal turns $G$ into a graph satisfying ${\mathcal P}$.
Yannakakis remarked already in \cite{Yannakakis1} that
{\em ``edge-deletion problems do not seem to be amenable in general to a unified approach. It would be interesting to find classes of properties for which this is possible, that is, classes of properties for which the edge-deletion problem can be shown NP-complete using a small number of reductions, or classes of properties for which there is a uniform polynomial algorithm that solves the edge-deletion problem.''}
Short of fully resolving Yannakakis's problem in the form of a precise characterization of the graph properties for which $\rem_{\mathcal P}(G)$ is NP-hard, one would at least like to answer this question for a natural subclass of properties. Perhaps the
most natural such family is the one consisting of all {\em monotone} graph properties, that is, properties closed under vertex and edge removal\footnote{Indeed, essentially all the properties studied in \cite{Yannakakis0,Yannakakis1} are monotone.}.
We thus advocate the study of the following special case of Yannakakis's problem.

\begin{problem}\label{prob1}
Characterize the monotone properties for which computing $\rem_{\mathcal P}(G)$ is NP-hard.
\end{problem}

Since Problem \ref{prob1} seems out of reach at the moment,
it is reasonable to focus on a very natural subfamily of monotone properties, where we fix a graph $H$ and define ${\mathcal P}_H$ to be
the property of being $H$-free. To simplify the notation, we use $\rem_{H}(G)$ instead of the more appropriate $\rem_{{\mathcal P}_H}(G)$.
Though it might seem counterintuitive at first, proving that computing $\rem_{H}(G)$ is NP-hard becomes {\em easier} when $H$ is more ``complicated'' (see Subsection \ref{subsec:previous} for more details).
Indeed, the first result concerning this problem was obtained in the 80's by Asano and Hirata \cite{AH,Asano}, who proved that
$\rem_{H}(G)$ is NP-hard whenever $H$ is $3$-connected. Another result was obtained by Alon, Shapira and Sudakov \cite{ASS},
who proved that computing $\rem_{H}(G)$ is NP-hard whenever $H$ is not bipartite. They left open the problem of characterizing the bipartite graphs $H$ for which this task is NP-hard.
In this paper we introduce a new approach
for resolving this problem. The relevant background is given in the next subsection.

\subsection{Background on the Erd\H{o}s-S\'os conjecture}

One of the oldest and most well-studied topics in graph theory is Tur\'an's problem,
which asks, for a fixed graph $H$, what is the maximum number of edges in an $n$-vertex $H$-free
graph. Denoting this quantity by $\mathrm{ex}(n,H)$, Tur\'an's theorem states that if $H=K_{t+1}$ (i.e., the complete graph on $t+1$ vertices) then $\mathrm{ex}(n,K_{t+1})\sim (1-1/t)\frac{n^2}{2}$. In fact, the following stronger statement holds; let $T_{t,n}$ denote
the complete $t$-partite graph with all parts of size $\lceil n/t \rceil$ or $\lfloor n/t \rfloor$ (this graph is called the {\em Tur\'an graph}).
Then $\mathrm{ex}(n,K_{t+1})$ is exactly the number of edges of $T_{t,n}$, and moreover, $T_{t,n}$ is the unique $K_{t+1}$-free graph
on $n$ vertices with this number of edges.
As we mentioned in the previous subsection, Alon, Shapira and Sudakov \cite{ASS} proved that computing
$\rem_{H}(G)$ is NP-hard for every non-bipartite $H$. One of their main tools was a strengthened version
of Tur\'an's theorem \cite{AES}. The main obstacle that prevents one from extending their approach to bipartite
$H$, is that their reduction produces graphs with $\Theta(n^2)$ edges, and so it inevitably creates new copies
of every fixed bipartite graph $H$.

Our main idea in this paper is that in order to determine the hardness of computing $\rem_{H}(G)$ for bipartite
$H$, one should consider another Tur\'an-type problem. The Erd\H{o}s-S\'{o}s conjecture \cite{Erdos_Sos_Conjecture} (see also \cite{CG}), which is one of the oldest problems in extremal graph theory\footnote{We should point out that
a solution to this problem for large enough trees $T$ was announced more than $10$ years by Ajtai, Koml\'os,
Simonovits and Szemer\'edi, but this result has yet to be published.} states that $\mathrm{ex}(n,T) \leq  \frac{k-2}{2}\cdot n$ for every tree $T$ on $k$ vertices. Observe that a simple example meeting this upper bound is the disjoint union of complete graphs on $k-1$ vertices. This motivates the following definition.

\begin{definition}\label{conj:es}
A $k$-vertex tree $T$ satisfies the {\em Strong Erd\H{o}s-S\'{o}s conjecture (SESC)} if the only $T$-free graph with $\frac{k-2}{2}\cdot n$ edges is the disjoint union of cliques of size $k-1$.
\end{definition}

It is easy to see that every tree satisfying the SESC, also satisfies the Erd\H{o}s-S\'{o}s conjecture.
Note that although it is clear that stars do satisfy the Erd\H{o}s-S\'{o}s conjecture, a star with $k-1 \geq 3$ leaves does {\em not} satisfy the SESC, since every $(k-2)$-regular graph has no copy of this star. Finally, we remark that a famous
theorem of Erd\H{o}s and Gallai \cite{ErdosGallai} states that paths satisfy the SESC.

\subsection{Our main results}

Our main graph-theoretic result in this paper is the following contribution to the study of the Erd\H{o}s-S\'{o}s conjecture.

\begin{theorem}\label{thm:Erdos_Sos_diameter_4_trees}
Every tree of diameter at most $4$ which is not a star satisfies the Strong Erd\H{o}s-S\'{o}s conjecture.
\end{theorem}

On a high level, our proof of Theorem \ref{thm:Erdos_Sos_diameter_4_trees} follows the strategy used by 
\cite{McLennan} to prove that trees of diameter at most $4$ satisfy
the Erd\H{o}s-S\'os conjecture. However, there are various subtle aspects that need to be changed in order to obtain
our stronger Theorem \ref{thm:Erdos_Sos_diameter_4_trees}. Perhaps the most significant one is a randomized process which
is a crucial ingredient enabling us to characterize the extremal graphs.

Our main complexity-theoretic result in this paper employs Theorem \ref{thm:Erdos_Sos_diameter_4_trees} to obtain
the following contribution to the study of Yannakakis's problem.

\begin{theorem}\label{thm:tree}
For every tree $T$ which is not a star, computing $\rem_{T}(G)$ is $\mathrm{NP}$-hard.
\end{theorem}

By a well-known result of Tutte \cite{Tutte}, if $T$ is a star then one can compute $\rem_{T}(G)$ in polynomial time using a reduction to the maximum matching problem. Tutte's argument in fact gives the following more general result:
\begin{theorem}[\cite{Tutte}]\label{thm:f_factor}
	There is a polynomial-time algorithm which receives as input a graph $G$ and a function $f : V(G) \rightarrow \{0,\dots,v(G)\}$, and computes the maximum number of edges in a spanning subgraph $F$ of $G$ with the property that $d_F(v) \leq f(v)$ for every $v \in V(G)$.
\end{theorem}
\noindent
For completeness, we give the (folklore) proof of Theorem \ref{thm:f_factor} in the appendix.

By the above, Theorem \ref{thm:tree} has the following immediate
corollary.


\begin{corollary}\label{coro:tree}
For a tree $T$, computing $\rem_{T}(G)$ is polynomial-time solvable if $T$ is a star and $\mathrm{NP}$-hard otherwise.
\end{corollary}

Next, we consider general forests and obtain a classification similar to Corollary \ref{coro:tree}. We say that a forest $F$ is a {\em star forest} if every connected component of $F$ is a star.

\begin{theorem}\label{thm:forest}
	For a forest $F$, computing $\rem_{F}(G)$ is polynomial-time solvable if $F$ is a star forest and $\mathrm{NP}$-hard otherwise.
\end{theorem}


The above theorem motivates us to conjecture that the answer to the question posed by Alon, Shapira and Sudakov \cite{ASS}
is the following.

\begin{conjecture}
Computing $\rem_H(G)$ is NP-hard if and only if $H$ is not a star forest.  
\end{conjecture}

\subsection{Comparison to previous results}\label{subsec:previous}

It is natural to try and prove that computing $\rem_H(G)$ is NP-hard using a reduction from vertex cover, as follows.
Given an $n$-vertex graph $G$ we construct an input $G'$ to $\rem_H(G')$ as follows:
First, we put in $G'$ a star with $n$ edges, each corresponding to one of the vertices of $G$.
Now, for every edge $(i,j) \in E(G)$ we add to $G'$ a copy of $H$ which contains the two edges of the star corresponding
to the vertices $i$ and $j$, and is otherwise disjoint from all other edges/vertices of $G'$. Let us call the copies of $H$ we added to $G'$ the {\em canonical copies}. While it is easy to see that if $\rem_H(G') \leq k$ then $G$ has a vertex cover of size at most $k$, the other
direction seems harder to prove. The main difficultly is in ensuring that after removing a set of edges from $G'$ which destroys
all the canonical copies of $H$, we do not
still end up with a copy of $H$ resulting from various ``pieces'' of canonical copies of $H$'s that together create a copy of $H$. 
It is thus clear that it should be easier to prove that the above reduction works when $H$ is more ``complicated''.
And indeed, it is not hard to see that this reduction works whenever $H$ is $3$-connected (this is basically the proof of \cite{AH,Asano}).
Hence, in some sense, the graphs which are hardest for the above approach are forests. This explains why we have to use a completely different
approach for handing them. Let us finally mention that the reduction of \cite{ASS} can only handle non-bipartite $H$ since it inherently produces graphs with $\Theta(n^2)$ edges. This approach fails for bipartite $H$ since graphs with $\Theta(n^2)$ edges cannot be $H$-free when $H$ is bipartite.

\paragraph{Paper overview:} In the next section we show how to derive Theorem \ref{thm:tree} from Theorem \ref{thm:Erdos_Sos_diameter_4_trees}.
We prove Theorem \ref{thm:tree} in Section \ref{sec:erdossos}. In Section \ref{sec:star forest}, we prove the positive direction of Theorem \ref{thm:forest} by giving a polynomial-time algorithm for computing $\rem_{F}(G)$ for a star forest $F$. In Section \ref{sec:disconnected} we prove the negative direction of Theorem \ref{thm:forest}.
Section \ref{sec:concluding} contains
some concluding remarks and open problems.

\paragraph{Notation:} For graphs $G,H$, denote by $\ex(G,H)$ the largest number of edges in an $H$-free subgraph of $G$. So $\ex(G,H) = e(G) - \rem_{H}(G)$. Hence, the problems of computing $\ex(G,H)$ and $\rem_H(G)$ are equivalent, and sometimes it will be convenient to consider $\ex(G,H)$ instead of $\rem_H(G)$.

\section{Deriving Theorem \ref{thm:tree} from Theorem \ref{thm:Erdos_Sos_diameter_4_trees}}\label{sec:proofmain}

We first claim that computing $\rem_{T}(G)$ is NP-hard for every tree $T$ on at least $4$ vertices satisfying the Strong Erd\H{o}s-S\'{o}s conjecture.
Indeed, suppose $n$ is divisible by $k-1$ and $G$ is an $n$-vertex graph with $m$ edges. Then $\rem_{T}(G) = m-{k-1 \choose 2}\frac{n}{k-1}$
if and only if $G$ contains a $K_{k-1}$-factor, that is, a collection of $\frac{n}{k-1}$ cliques of size $k-1$ covering all its vertices.
Since deciding whether a graph has a $K_{k-1}$-factor is well-known to be NP-hard for $k \geq 4$ \cite{KH}, we conclude that computing $\rem_{T}(G)$ is NP-hard for such $T$.

We now prove Theorem \ref{thm:tree} by induction on $|V(T)|$. Note that by the previous paragraph and Theorem \ref{thm:Erdos_Sos_diameter_4_trees}, we already know that computing $\rem_{T}(G)$ is NP-hard for every tree of diameter at most $4$ which is not a star\footnote{Note that such a tree $T$ must satisfy $|T|=k \geq 4$, and so the NP-hardness of finding a $(k-1)$-factor applies.}. These trees will form the base of our induction\footnote{Obviously, every (non-star) tree on $4$ vertices has diameter at most $4$.}.
Consider now a tree $T$ of diameter at least $5$ and let $T_0$ be the tree obtained by removing all the leaves of $T$.
It is easy to see that since $T$ has diameter at least $5$, then $T_0$ has diameter at least $3$ (i.e., $T_0$ is not a star), implying by induction that computing  $\rem_{T_0}(G)$ is NP-hard.
We will now show that computing $\rem_{T_0}(G)$ can be reduced to computing $\rem_{T}(G)$, which will complete the induction step and thus the proof of the theorem.

Given an $n$-vertex graph $G$ as an input to $\rem_{T_0}(G)$, let $G'$ be the graph obtained from $G$ by doing the following: for each $v \in V(G)$, add to $G$ a set $L_v$ of $\binom{n}{2} + |V(T)|$ new vertices and connect all of them to $v$ (for distinct $v,v' \in V(G)$, the sets $L_v,L_{v'}$ are disjoint). We claim that $\rem_{T_0}(G)=\rem_{T}(G')$. Indeed, suppose $E$ is a set of edges whose removal turns
$G$ into a $T_0$-free graph, and consider the graph $G'- E$.
Clearly $G' - E$ has no copy of $T_0$
contained within the original vertices of $G$. Furthermore, since each of the new vertices we added to $G$ has degree $1$, the graph $G' - E$ has no copy of $T$.
We deduce that $\rem_{T}(G') \leq \rem_{T_0}(G)$. We now claim that removing from $G'$ less than $\rem_{T_0}(G)$ edges cannot make it $T$-free.
Indeed, take any set $E'$ of less than $\rem_{T_0}(G)$ edges and consider $G'- E'$. Since $|E'| < \rem_{T_0}(G)$ we know that
$G' - E'$ still has a copy of $T_0$ on the original vertices of $G$. Also, since $|E'| < \rem_{T_0}(G) \leq \binom{n}{2}$,
in the graph $G' - E'$ every vertex of $G$ still touches at least $|V(T)|$ of the new edges that were connected to it. Hence we can extend
the copy of $T_0$ into a copy of $T$. We have thus completed the proof that $\rem_{T_0}(G)=\rem_{T}(G')$.

\section{Proof of Theorem \ref{thm:Erdos_Sos_diameter_4_trees}}\label{sec:erdossos}

We will need the following lemma, which is implicit in \cite{McLennan}. For completeness, we include its proof. For an $n$-vertex graph $G$, we use $d(G)$ to denote the average degree of $G$, namely $d(G) = 2e(G)/n$.

\begin{lemma}\label{lem:high_degree_neighbours}
	For every $t \in [0,1]$ and for every graph $G$, there is a vertex
	$u \in V(G)$ such that
	$$
	\sum_{v \in N(u)}{\left( 1 - \frac{t \cdot d(G)}{d(v)} \right)} \geq (1-t) \cdot d(G).
	$$
\end{lemma}

\begin{proof}
Note that $\sum_{u \in V(G)}\sum_{v \in N(u)}{\frac{1}{d(v)}} =
\sum_{v \in V(G)}{1} = n$. By using this we get
\begin{align*}
\sum_{u \in V(G)}\sum_{v \in N(u)}{\left( 1 - \frac{t \cdot d(G)}{d(v)} \right)} &=
\sum_{u \in V(G)}{d(u)} -
t \cdot d(G) \cdot \sum_{u \in V(G)}\sum_{v \in N(u)}{\frac{1}{d(v)}}
\\ &=
n \cdot d(G) - t \cdot d(G) \cdot n
\\ &=
n \cdot (1-t) \cdot d(G).
\end{align*}
By averaging, there must be some $u \in V(G)$ for which the assertion holds.
\end{proof}

\noindent
We will also need the following lemma.


\begin{lemma}\label{lem:number_partition_tree_embedding}
Let $p \geq 1, \ell \geq 2$, and $\gamma_1 \leq \dots \leq \gamma_p$ be nonnegative integers and let $s_1,\dots,s_{\ell}$ be integers satisfying $s_i \geq \gamma_p+1$ for every $1 \leq i \leq \ell$ and $s_1 + \dots + s_{\ell} \geq \sum_{i=1}^p{(1+\gamma_i)} + (\ell-1) \gamma_{p-1}$, where $\gamma_{p-1}$ is interpreted as $0$ if $p=1$.
Then there is a partition $[p] = J_1 \cup \dots \cup J_{\ell}$ such that
	$s_i \geq \sum_{j \in J_i}{(1+\gamma_j)}$ for every
	$1 \leq i \leq \ell$.
\end{lemma}

\begin{proof}
We proceed by induction on $\ell$.
	It will be convenient to prove the base case and the induction step simultaneously.
	Let $1 \leq q \leq p$ be the minimal integer satisfying
	$s_{\ell} \geq \sum_{i=q}^{p}{(1 + \gamma_i)}$. Note that $q$ is well-defined since by assumption we have $s_{\ell} \geq 1 + \gamma_p$.
	If $q = 1$ then we are done, as we can choose $J_{\ell} = [p]$, and $J_i = \emptyset$ for all $1 \leq i \leq \ell - 1$. So suppose that $q > 1$. By the minimality of $q$ we have
	$
	s_{\ell} < \sum_{i=q-1}^{p}{(1+\gamma_i)}
	$
	and hence
	$s_{\ell} \leq \sum_{i=q}^{p}{(1+\gamma_i)} + \gamma_{q-1}$.
	Therefore,
	\begin{align*}
	s_1 + \dots + s_{\ell-1} &\geq \sum_{i = 1}^{p}{(1+\gamma_i)} + (\ell-1)\gamma_{p-1} - s_{\ell}
	\geq
	 \sum_{i = 1}^{p}{(1+\gamma_i)} + (\ell-1)\gamma_{p-1} - \sum_{i=q}^{p}{(1+\gamma_i)} - \gamma_{q-1}
	\\ &\geq
	\sum_{i = 1}^{q-1}{(1+\gamma_i)} + (\ell-2)\gamma_{p-1}
	\geq
	\sum_{i = 1}^{q-1}{(1+\gamma_i)} + (\ell-2)\gamma_{q-2}.
	\end{align*}
	For $\ell = 2$ the above gives $s_1 \geq \sum_{i = 1}^{q-1}{(1 + \gamma_i)}$,
	so the assertion of the lemma holds with
	$J_1 = \{1,\dots,q-1\}$ and $J_2 = \{q,\dots,p\}$.
	
	For $\ell \geq 3$, we have
	$s_1 + \dots + s_{\ell-1} \geq
	\sum_{i = 1}^{q-1}{(1+\gamma_i)} + (\ell-2)\gamma_{q-2}$, allowing us to apply the induction hypothesis for $\ell-1$, $\gamma_1,\dots,\gamma_{q-1}$ and $s_1,\dots,s_{\ell-1}$. We thus obtain a partition $[q-1] = J_1 \cup \dots \cup J_{\ell-1}$ such that $s_i \geq \sum_{j \in J_i}{(1 + \gamma_j)}$ for each $1 \leq i \leq \ell-1$. Setting $J_{\ell} = \{q,\dots,p\}$ completes the proof.
\end{proof}

We are now ready to prove Theorem \ref{thm:Erdos_Sos_diameter_4_trees}.

\begin{proof}[Proof of Theorem \ref{thm:Erdos_Sos_diameter_4_trees}]
Fix a tree $T$ on $k$ vertices of diameter at most $4$ which is not a star.
It is easy to see that there is a vertex\footnote{A middle vertex of a longest path in $T$, whose length is at most $4$, satisfies the required properties.}
$a \in V(T)$ such that $d_T(a) \geq 2$ and such that any vertex of $T$ is at distance at most $2$ from $a$.
Let $b_1,\dots,b_p$ be the neighbors of $a$ in $T$. For each $1 \leq i \leq p$, let $C_i$ be the set of neighbors of $b_i$ excluding $a$, and put $\gamma_i = |C_i|$. Note that
	$V(T) = \{a\} \cup \{b_1,\dots,b_p\} \cup \bigcup_{i=1}^{p}{C_i}$ (where the union is disjoint) and hence
\begin{equation}\label{eq:sum of gamma_i}
\sum_{i=1}^{p}{\gamma_i} = k - p - 1,
\end{equation}
and that all vertices in $\bigcup_{i=1}^{p}{C_i}$ are leaves of $T$. Moreover, since $T$ is not a star there must be some $1 \leq i \leq p$ for which $\gamma_i > 0$.
Suppose, without loss of generality, that $\gamma_1 \leq \dots \leq \gamma_p$, and so $\gamma_p \geq 1$.

Suppose now that $G$ is an $n$-vertex $T$-free graph with at least $\frac{k-2}{2} \cdot n$ edges. We need to prove that $G$ must be a union of cliques of size $k-1$. It suffices to show that if $G$ is connected then it is a clique\footnote{Indeed, if $G$ has $r \geq 2$ connected components of sizes $n_1, \dots, n_r$,
then some component $i$ has at least $\frac{k-2}{2} \cdot n_i$ edges, and is $T$-free.
Thus, by the claim for connected graphs, component $i$ is a clique of size $k-1$.
The assertion that $G$ is a union of cliques now follows by induction.}
of size $k-1$.
Apply Lemma \ref{lem:high_degree_neighbours} to $G$ with $t = \frac{k-p-1}{k-2}$ to get a vertex $u \in V(G)$ with
	\begin{equation}\label{eq:high_degree_neighbours}
	\sum_{v \in N(u)}{\left( 1 - \frac{k-p-1}{d(v)} \right)} \geq
	p-1.
	\end{equation}
	Above we used our assumption that $d(G) \geq k-2$.
	Let $W$ be the set of vertices $v \in N(u)$ with $d(v) \geq k-p$. If $d(v) \leq k-p-1$ then $v$ has a non-positive contribution to the sum in \eqref{eq:high_degree_neighbours}. Hence,
	\begin{equation}\label{eq:high_degree_neighbours 2}
		\sum_{v \in W}{\left( 1 - \frac{k-p-1}{d(v)} \right)} \geq
		p-1.
	\end{equation}

	For an ordering $\sigma = (v_1,\dots,v_m)$ of the vertices in $W$, let $I = I_{\sigma}$ be the set of all $v_i$ ($1 \leq i \leq m$) such that $v_i$ has at least $k-p-1$ neighbors which are {\em not} in $\{u,v_1,\dots,v_i\}$.
	We claim that if $|I| \geq p$ then $G$ contains a copy of $T$, thus contradicting our assumption. Indeed, assuming $|I| \geq p$ and fixing $p$ vertices $v_{i_1},\dots,v_{i_p} \in I$ with $1 \leq i_1 < \dots < i_p \leq m$, we can embed $T$ in $G$ as follows. We let $u$ play the role of $a$, and let $v_{i_j}$ play the role of $b_j$ for each $1 \leq j \leq p$. We then choose a set
	$S_1 \subseteq N(v_{i_1}) \setminus
	\left( \{u,v_1,\dots,v_{i_1}\} \cup \{v_{i_2},\dots,v_{i_p}\} \right)$ of size $\gamma_1$ to play the role of $C_1$; a set
	$S_2 \subseteq N(v_{i_2}) \setminus
	\left( \{u,v_1,\dots,v_{i_2}\} \cup \{v_{i_3},\dots,v_{i_p}\} \cup S_1 \right)$ of size $\gamma_2$ to play the role of $C_2$, and so on; at the last step we choose a set
	$S_p \subseteq N(v_{i_p}) \setminus (\{u,v_1,\dots,v_{i_p}\} \cup
	S_1 \cup \dots \cup S_{p-1})$ of size $\gamma_p$ to play the role of $C_p$. Let us explain why the choice of $S_j$ is possible for each $1 \leq j \leq p$. If $\gamma_j = 0$ then the assertion is trivial. So suppose that $\gamma_j \geq 1$. Then $\gamma_{j+1},\dots,\gamma_p \geq 1$ as well. We have
\begin{align*}
\left|
\{v_{i_{j+1}},\dots,v_{i_p}\} \cup S_1 \cup \dots \cup S_{j-1}
\right| &= p-j + \gamma_1 + \dots + \gamma_{j-1} =
p-j + k-p-1 - (\gamma_j + \dots + \gamma_p) \\&\leq
k - p - 1 - \gamma_j,
\end{align*}
where the second equality uses \eqref{eq:sum of gamma_i}.
By the definition of $I$, we know that $v_{i_j}$ has at least $k - p - 1$ neighbors not in $\{u,v_1,\dots,v_{i_j}\}$, hence we can always choose a set
$$S_j \subseteq N(v_{i_j}) \setminus (\{u,v_1,\dots,v_{i_j}\} \cup \{v_{i_{j+1}},\dots,v_{i_p}\} \cup S_1 \cup \dots \cup S_{j-1})$$ of size $\gamma_j$. This gives an embedding of $T$ into $G$.
	
	We have thus shown that $|I_{\sigma}| < p$ for every ordering $\sigma$ of $W$. Now choose such an ordering $\sigma$ at random. We claim that for every $v \in W$,
	\begin{equation*}\label{eq:prob v in I}
	\mathbb{P}[v \in I_{\sigma}] \geq 1 - \frac{k-p-1}{d(v)},
	\end{equation*}
	with equality only if $N(v) \subseteq \{u\} \cup W$.
	So fix any $v \in W$ and set $d_1 = |N(v) \cap W|$ and $d_2 = |N(v) \setminus (\{u\} \cup W)|$. Then $d_1 + d_2 = d(v) - 1$.
	If $d_2 \geq k-p-1$ then
	$\mathbb{P}[v \in I] = 1 > 1 - \frac{k-p-1}{d(v)}$, as required. Suppose now that $d_2 < k-p-1$.
	For $v$ to be in $I_{\sigma}$, we need that among the $d_1+1$ vertices in $\{v\} \cup (N(v) \cap W)$, the random permutation would place at least $k - p - 1 - d_2$ vertices after $v$. The probability for this is
	$$
	\frac{d_1+d_2+1 - (k-p-1)}{d_1+1} = \frac{d(v) - (k - p - 1)}{d_1 + 1} \geq 1 - \frac{k - p - 1}{d(v)},
	$$
	with equality only if $d_1+1 = d(v)$, namely if all neighbors of $v$ are inside $\{u\} \cup W$. Here we use the fact that $d(v)\geq k-p$, as $v \in W$. This proves our claim.
	By linearity of expectation, we get
	\begin{equation}\label{eq:random_set_expected_size}
		\mathbb{E}[|I|] \geq
		\sum_{v \in W}{\left( 1 - \frac{k-p-1}{d(v)} \right)},
	\end{equation}
	with equality only if $N(v) \subseteq \{u\} \cup W$ for every $v \in W$.
By combining \eqref{eq:random_set_expected_size} with \eqref{eq:high_degree_neighbours 2}, we see that $\mathbb{E}[|I|] \geq p-1$. On the other hand, we saw that $|I| < p$. This means that $|I|$ is a constant random variable, attaining the value $p-1$ with probability $1$. In particular, we have equality in \eqref{eq:random_set_expected_size}. Hence, $N(v) \subseteq \{u\} \cup W$ for every $v \in W$.

We claim that $G[W]$ is the disjoint union of cliques. If not, then there are distinct vertices $v_1,v_2,v_3 \in W$ such that $\{v_1,v_2\},\{v_2,v_3\} \in E(G)$ but $\{v_1,v_3\} \notin E(G)$.
Fix $x_1,\dots,x_{k-p-3} \in N(v_2) \setminus \{u,v_1,v_3\}$; these exist because $d(v_2) \geq k-p$ by the definition of $W$. Also, $x_1,\dots,x_{k-p-3} \in W$.
Let $\sigma$ be an ordering of $W$ ending in $v_2,v_1,v_3,x_1,\dots,x_{k-p-3}$, and let $\tau$ be the ordering obtained from $\sigma$ by swapping $v_1,v_2$; namely, $\tau$ ends in $v_1,v_2,v_3,x_1,\dots,x_{k-p-3}$. It is clear that $I_{\sigma} \setminus \{v_1,v_2\} = I_{\tau} \setminus \{v_1,v_2\}$. Also, $v_2 \in I_{\sigma}$, as it has $k-p-1$ neighbors following it in $\sigma$, but $v_2 \notin I_{\tau}$ because it has only $k-p-2$ neighbors following it in $\tau$. Moreover, $v_1 \notin I_{\sigma},I_{\tau}$ because $v_1$ is not adjacent to $v_3$.
It follows that $I_{\tau} = I_{\sigma} \setminus \{v_2\}$, contradicting the fact that $|I|$ is a constant random variable.

We have thus shown that $G[W]$ is the disjoint union of cliques. Denote these cliques by $S_1,\dots,S_{\ell}$. We now show that $\ell = 1$ and $|W| = k-2$.
We have
$|S_1|,\dots,|S_{\ell}| \geq k-p = \gamma_1 + \dots + \gamma_p + 1 \geq \gamma_p + 1$ by the definition of $W$ and by the fact that $N(v) \subseteq \{u\} \cup W$ for every $v \in W$.
Also, since $G[W]$ is a disjoint union of cliques, for every $\sigma$ \nolinebreak we \nolinebreak have
$$
p - 1 = |I_{\sigma}| = \sum_{i = 1}^{\ell}{\big(|S_i| - (k - p - 1)\big)} =
\sum_{i = 1}^{\ell}{|S_i|} - \ell(k - p - 1),
$$
and hence
	\begin{equation}\label{eq:size_of_V1}
	|W| = |S_1| + \dots + |S_{\ell}| = \ell(k-p-1) + p - 1 = \ell(\gamma_1 + \dots + \gamma_p) + p-1.
	\end{equation}
	Suppose by contradiction that $\ell \geq 2$.
	Recall that $\gamma_p \geq 1$. By \eqref{eq:size_of_V1}, we have
	$$
	|S_1| + \dots + |S_{\ell}| \geq  \sum_{i=1}^{p}{(1+\gamma_i)} - 1 + (\ell-1) \cdot (\gamma_{p-1}+\gamma_p) \geq \sum_{i=1}^{p}{(1+\gamma_i)} + (\ell-1)\gamma_{p-1}.$$
	So we can apply Lemma \ref{lem:number_partition_tree_embedding} with $s_i = |S_i|$ ($1 \leq i \leq \ell$) to obtain a partition $[p] = J_1 \cup \dots \cup J_{\ell}$ such that $|S_i| \geq \sum_{j \in J_i}{(1 + \gamma_j)}$. But this means that we can embed $T$ into $G$ by mapping $a$ to $u$ and
	$\bigcup_{j \in J_i}{(\{b_j\} \cup C_j)}$ to $S_i$ for every $1 \leq i \leq \ell$. This is a contradiction, and hence $\ell = 1$. This means that $W$ is a clique. Moreover, plugging $\ell = 1$ into \eqref{eq:size_of_V1} gives $|W| = k-2$. Now, $W \cup \{u\}$ is a clique of size $k-1$. This means that $G$ has no vertices other than $W \cup \{u\}$, as otherwise, by the connectivity of $G$, there would be a vertex outside $W \cup \{u\}$ adjacent to a vertex in $W \cup \{u\}$, and then we could embed $T$ into $G$ by using this vertex and $W \cup \{u\}$ (we could in fact embed any $k$-vertex tree in this case). So $G$ is indeed a clique of size $k-1$, completing the proof.
\end{proof}

\section{A Polynomial Algorithm for Star Forests}\label{sec:star forest}
In this section we prove the positive direction of Theorem \ref{thm:forest}, which we rephrase as follows.
\begin{theorem}\label{thm:star_forest}
	For every star forest $H$, there is a polynomial-time algorithm that computes $\ex(G,H)$.
\end{theorem}

Let us introduce some notation that we will use throughout the section.
Let $H$ be a star forest.
We may assume that $H$ has no isolated vertices; indeed, if we let $K$ be the graph obtained from $H$ by deleting all isolated vertices, then for every graph $G$ on at least $v(H)$ vertices, we have $\ex(G,H) = \ex(G,K)$.
Let $S_1,\dots,S_r$ denote the components of $H$, let $t_i$ be the number of leaves in $S_i$, and assume that $t_1 \geq \dots \geq t_r \geq 1$. 
For each $1 \leq i \leq r$, we denote by $H_i$ the star forest whose connected components are $S_1,\dots,S_i$ (so $H_r = H$). It will be convenient to denote the empty graph by $H_0$.
For a graph $G$, we use $\Delta(G)$ to denote the maximum degree of \nolinebreak $G$.
\begin{lemma}\label{lem:low_max_degree}
	Let $F$ be an $H$-free graph, and let $0 \leq i \leq r-1$ be the largest integer for which $F$ contains a copy of $H_i$. Then $F$ contains at most $(\Delta(F)+1)v(H)$ vertices of degree at least $t_{i+1}$.
\end{lemma}
\begin{proof}
	Suppose first that $r=1$, and so $i=0$. Then $H=S_1=K_{1,t_1}$, and $F$ contains no vertices of degree at least $t_1$, because $F$ is $H$-free.
	Suppose now that $r \geq 2$.
	By assumption, $F$ contains a copy of $H_i$.
	Let $X$ be the vertex-set of such a copy.
	Let $y_1,\dots,y_{\ell}$ be the vertices in $V(F) \setminus X$ which have degree at least $t_{i+1}$ in $F$. 
	Since $F$ is $H_{i+1}$-free, each $y_i$ must have a neighbour in $X$. By averaging, there is a vertex $x \in X$ adjacent to at least $\ell/|X|$ of the vertices $y_1,\dots,y_{\ell}$. Hence, $\ell \leq |X| \cdot \Delta(F)$. 
	So the number of vertices of $F$ of degree at least $t_{i+1}$ is at most
	$|X| + \ell = v(H_i) + \ell \leq (\Delta(F)+1)v(H)$.
\end{proof}

In the following lemma, we prove Theorem \ref{thm:star_forest}
in the case that the input graph has bounded maximum degree.
\begin{lemma}\label{lem:star-forest_low_max_degree}
	For each constant $C > 0$ there is an algorithm which runs in time $n^{O(v(H) \cdot C)}$ and computes $\ex(G,H)$ for input graphs $G$ with $\Delta(G) \leq C$.
\end{lemma}

\begin{proof}
	The algorithm works as follows: Go over all sets $U \subseteq V(G)$ of size at most $(C+1)v(H)$. For each such $U$, go over all (spanning) subgraphs $F'$ of $G$ in which $V(G) \setminus U$ is an independent set, and do the following:
	\begin{enumerate}
		\item Check whether $F'$ is $H$-free. If not, continue to the next graph.
		\item Find the maximum $0 \leq i \leq r-1$ such that $H_i$ is a subgraph of $F'$.
		\item Check whether $d_{F'}(w) \leq t_{i+1}-1$ for every $w \in V(G) \setminus U$. If not, continue to the next graph.
		\item Let $f : V(G)\setminus U \rightarrow \mathbb{N}$ be the function $f(w) = t_{i+1} - 1 - d_{F'}(w) \geq 0$. Use Theorem \ref{thm:f_factor} to compute the maximum number of edges $m$ in a (spanning) subgraph $F''$ of $G[V(G) \setminus U]$ with the property that $d_{F''}(w) \leq f(w)$ for every $w \in V(G) \setminus U$. Define $M(U,F') := e(F') + m$.
	\end{enumerate}
	Output the maximum of $M(U,F')$ over all $U,F'$ as above (which pass the tests in Items 1 and 3).
	
	Before proving the correctness of the above algorithm, let us consider its running time. It is easy to see that for each choice of $U,F'$ (as described above), we can execute steps 1-4 in time $n^{O(v(H))}$. Also, the number of choices of $U$ is
	$\sum_{i=0}^{(C+1)v(H)}\binom{n}{i} \leq n^{O(v(H) \cdot C)}$, and for each choice of $U$ there are only $O(1)$ choices for $F'$, because $|U| = O(1)$, $\Delta(G) \leq C = O(1)$, and we require that $V(G) \setminus U$ is an independent set in $F'$. So the running time of the algorithm is indeed $n^{O(v(H) \cdot C)}$, as required.
	
	Let us now show that the above algorithm correctly computes $\ex(G,H)$. First we show that for every $U,F'$ which pass the tests in Items 1 and 3, it holds that $M(U,F') \leq \ex(G,H)$.
	So let $U$ be a subset of $V(G)$ of size at most $(C+1)v(H)$ and let $F'$ be an $H$-free subgraph of $G$ such that
	$F'[V(G) \setminus U]$ is an independent set. Let $0 \leq i \leq r-1$ be the largest integer such that $H_i$ is a subgraph of $F'$ (as in Item 2). Suppose that $d_{F'}(w) \leq t_{i+1}-1$ for every $w \in V(G) \setminus U$, and set $f(w) = t_{i+1} - 1 - d_{F'}(w)$. Let $F''$ be a subgraph of $G[V(G) \setminus U]$ satisfying $d_{F''}(w) \leq f(w)$ for all $w \in V(G) \setminus U$, and having the maximum number of edges among all subgraphs with this property.
	Then $M(U,F') = e(F') + e(F'')$.
	Let $F$ be the union of $F'$ and $F''$. Then $d_F(w) = d_{F'}(w) + d_{F''}(w) \leq t_{i+1}-1$ for each $w \in V(G) \setminus U$. This implies that for every $1 \leq j \leq i+1$, there is no copy of $K_{1,t_j}$ in $F$ whose center is in $V(G) \setminus U$. Hence, every copy of $H_{i+1}$ in $F$ is also contained in $F'$. But $F'$ is $H_{i+1}$-free by our choice of $i$, so $F$ is $H_{i+1}$-free and hence also $H$-free. So we see that $M(U,F') = e(F') + e(F'') = e(F) \leq \ex(G,H)$. This shows that the value outputted by the algorithm is at most $\ex(G,H)$.
	
	For the other direction, let $F$ be an $H$-free subgraph of $G$ such that $e(F) = \ex(G,H)$. Let $0 \leq i \leq r-1$ be the largest integer for which $H_i$ is a subgraph of $F$, and let $U$ be the set of all vertices $v \in V(G)$ satisfying $d_F(v) \geq t_{i+1}$. By Lemma \ref{lem:low_max_degree} and the assumption
	$\Delta(G) \leq C$, we have that $|U| \leq (\Delta(F) + 1)v(H) \leq (C+1)v(H)$.
	Let $F'$ be the subgraph of $F$ obtained by deleting all edges of $F$ which are contained in $V(G) \setminus U$. Clearly, $F' \subseteq F$ is $H_{i+1}$-free, and hence $H$-free. Also, $d_{F'}(w) \leq d_F(w) \leq t_{i+1}-1$ for every $w \in V(G) \setminus U$. By the choice of $i$, $F$ contains a copy of $H_i$. Also, for each $1 \leq j \leq i$, there is no copy of $K_{1,t_j}$ in $F$ whose center is in $V(G) \setminus U$. Therefore, this copy of $H_i$ is also contained in $F'$. It follows that $i$ is also the largest integer for which $F'$ contains a copy of $H_i$.
	Hence, the pair $(U,F')$ passes the tests in Items 1 and 3. Now, setting $f(w) = t_{i+1}-1-d_{F'}(w)$ for each $w \in V(G) \setminus U$, we see that $F'' := F[V(G) \setminus U]$ is a subgraph of $G[V(G) \setminus U]$ satisfying $d_{F''}(w) \leq f(w)$ for each $w \in V(G) \setminus U$. This implies that $M(U,F') \geq e(F') + e(F'') = e(F) = \ex(G,H)$. So we see that the value outputted by the algorithm is at least $\ex(G,H)$. This completes the proof.
\end{proof}

To handle input graphs with large maximum degree, we need the following lemma. Let $H'$ denote the star forest with components $S_2,\dots,S_r$ (i.e. $H' = H - S_1$).

\begin{lemma}\label{prop:structured_extremal_subgraphs}
	There is $D = D(H)$ such that for every graph $G$ with $\Delta(G) \geq D$ and for every $H$-free spanning subgraph $F$ of $G$ satisfying $e(F) = \ex(G,H)$ and $\Delta(F) \geq t_1$, there is a vertex $v \in V(G)$ such that $F - v$ is $H'$-free.
\end{lemma}
\begin{proof}
	We prove the proposition with
	$$
	D = D(H) := d(d+1)v(H) + d,
	$$
	where $d := v(H) - 1 = v(H') + t_1$.
	Let $w \in V(G)$ be a vertex of maximum degree, i.e., $d_G(w) = \Delta(G)$.
	By assumption, $d_G(w) \geq D$.
	Let $F$ be an $H$-free (spanning) subgraph of $G$ satisfying $e(F) = \ex(G,H)$ and $\Delta(F) \geq t_1$. First we show that $\Delta(F) \geq d$. Suppose otherwise. Let $1 \leq i \leq r-1$ be the largest integer for which $H_i$ is a subgraph of $F$ (note that $i$ is well-defined since by assumption we have $H_1=S_1 \subseteq F$ as $\Delta(F) \geq t_1$).
	Let $$U = \{u \in V(G)\setminus \{w\} : d_F(u) \geq t_{i+1}\}.$$
	By Lemma \ref{lem:low_max_degree} we have
	$|U| \leq (\Delta(F) + 1)v(H) < (d+1)v(H)$.
	Now let $F'$ be the graph obtained from $F$ by deleting all edges incident to vertices of $U$, and then adding all edges of $G$ incident to \nolinebreak $w$. \nolinebreak Then
	$$
	e(F') > e(F) - d|U| + (d_G(w) - d) > e(F) - d(d+1)v(H) + (D - d) = e(F),
	$$
	where in the first two inequalities we used the assumptions $\Delta(F) < d$ and $d_G(w) \geq D$, and in the equality we used our choice of $D$.
	As $e(F') > e(F) = \ex(G,H)$, $F'$ must contain a copy of $H$. Note, however, that if
	$u \in V(G)$ satisfies $d_{F'}(u) \geq t_{i+1}$ then either $u = w$, or $d_{F'}(u) = t_{i+1}$ and $wu \in E(F')$. Hence, every copy of $K_{1,t_{i+1}}$ in $F'$ must contain $w$, so $F'$ does not contain two disjoint copies of $K_{1,t_{i+1}}$. But this means that $F'$ is $H$-free, because $H$ contains the disjoint union of $K_{1,t_i}$ and $K_{1,t_{i+1}}$ and $t_i \geq t_{i+1}$. This contradiction shows that $\Delta(F) \geq d$, as claimed. 
	
	Let $v \in V(F)$ with $d_F(v) \geq d$. Suppose by contradiction that $F - v$ contains a copy of $H'$. As $d_F(v) \geq d = v(H') + t_1$, we can find $t_1$ neighbors of $v$ which do no participate in this copy of $H'$. This gives a copy of $H$ in $F$, a contradiction. This completes the proof of the lemma.
\end{proof}

\begin{proof}[Proof of Theorem \ref{thm:star_forest}]
	The proof is by induction on $r$. The base case $r=1$ follows from Theorem \ref{thm:f_factor}. Suppose now that $r \geq 2$.
	Let $D = D(H)$ be the constant given by Lemma \ref{prop:structured_extremal_subgraphs}.
	The algorithm computes $\Delta(G)$ and proceeds as follows:
	\begin{itemize}
		\item If $\Delta(G) < D$ then compute $\ex(G,H)$ using the algorithm given by Lemma \ref{lem:star-forest_low_max_degree}.
		\item If $\Delta(G) \geq D$, compute $M_1 := \ex(G,K_{1,t_1})$ (namely, the maximum number of edges in a subgraph of $G$ with maximum degree at most $t_1-1$), and
		$M_2 := \max_{u \in V(G)}{\left[ d_G(u) + \ex(G - u, H') \right]}$. Output $M := \max\{M_1,M_2\}$. Note that $M_1$ can be computed in polynomial time by Theorem \ref{thm:f_factor}, and $M_2$ can be computed in polynomial time by the induction hypothesis for $r-1$.
	\end{itemize}
	Let us show that the above algorithm correctly computes $\ex(G,H)$. In the first item this is clearly the case. So suppose that $\Delta(G) \geq D$.
	Our goal is to show that $\ex(G,H) = M$.
	As every $K_{1,t_1}$-free graph is also $H$-free, it follows that $\ex(G,H) \geq M_1$.
	Now, for each $u \in V(G)$, let $F'_u$ be an $H'$-free subgraph of $G-u$ with $e(F'_u) = \ex(G-u,H')$. Let $F_u$ be the subgraph of $G$ consisting of $F'_u$ and all edges of $G$ touching $u$. So $e(F_u) = d_G(u) + \ex(G - u,H')$.
	Then $F_u$ is $H$-free, because $F'_u = F_u - u$ is $H'$-free.
	Hence, $\ex(G,H) \geq e(F_u)$. This shows that $\ex(G,H) \geq \max_{u \in V(G)}{e(F_u)} = M_2$.
	We conclude that $\ex(G,H) \geq \max\{M_1,M_2\} = M$.
	
	Next we show that $M \geq \ex(G,H)$. Let $F$ be a spanning $H$-free subgraph of $G$ with $e(F) = \ex(G,H)$. If $\Delta(F) \leq t_1-1$ then $e(F) \leq \ex(G,K_{1,t_1}) = M_1 \leq M$.
	Suppose now that $\Delta(F) \geq t_1$. Then by Lemma \ref{prop:structured_extremal_subgraphs}, there is $v \in V(G)$ such that $F - v$ is $H'$-free.
	This implies that $e(F) \leq d_G(v) + \ex(G-v,H') \leq M_2 \leq M$.
	We conclude that $\ex(G,H) = e(F) \leq M$, as required.
\end{proof}

\section{Hardness for Forests}\label{sec:disconnected}
In this section we prove the negative direction of Theorem \ref{thm:forest}, which we rephrase as follows.
\begin{theorem}\label{thm:disconnected}
	Let $H$ be a forest one of whose connected components is not a star. Then computing $\ex(G,H)$ is NP-hard.
\end{theorem}
\noindent
For a graph $G$ and an integer $k$, denote by $kG$ the disjoint union of $k$ copies of $G$.

\begin{lemma}\label{claim:kG}
	Let $C$ be a connected graph and $k \geq 2$.  Then for every graph $G$, it holds that
	$\ex(kG,kC) = \ex((k-1)G,kC) + \ex(G,C).$
\end{lemma}
\begin{proof}
	Denote by $G_1,\dots,G_k$ the disjoint copies of $G$ in $kG$. If we destroy all copies of $kC$ in $G_1 \cup \dots \cup G_{k-1}$ and all copies of $C$ in $G_k$, then the resulting graph is $kC$-free (as $C$ is connected). Hence
	$
	\ex(kG,kC) \geq \ex((k-1)G,kC) + \ex(G,C).
	$
	In the other direction, let $F$ be a $kC$-free subgraph of $kG$ satisfying $e(F) = \ex(kG, kC)$. Since $F$ is $kC$-free, there must be some $1 \leq i \leq k$ for which $F[V(G_i)]$ is $C$-free. Assume without loss of generality that $i = k$, namely that $F'' := F[V(G_k)]$ is $C$-free. Since $F' := F[V(G_1) \cup \dots \cup V(G_{k-1})]$ is clearly $kC$-free, we have
	$$
	\ex(kG,kC) = e(F) =
	e(F') + e(F'') \leq
	\ex((k-1)G,kC) + \ex(G,C),
	$$
	as required.
\end{proof}
It will be convenient to first prove Theorem \ref{thm:disconnected} in the case that all connected components of $H$ are isomorphic, namely that $H = kT$ for some tree $T$ and integer $k \geq 1$.
\begin{lemma}\label{lem:k_disjoint_copies}
	For every tree $T$ which is not a star, and for every integer $k \geq 1$, computing $\ex(G,kT)$ is NP-hard.
\end{lemma}
\begin{proof}
	Lemma \ref{claim:kG} shows that if we could compute $\ex(G,H) = \ex(G,kT)$ in polynomial time for every graph $G$, then we could also compute $\ex(G,T)$ in polynomial time for every graph $G$, as $\ex(G,T) = \ex(kG,kT) - \ex((k-1)G,kT)$. But computing $\ex(G,T)$ is NP-hard by Theorem \ref{thm:tree}.
\end{proof}
\begin{proof}[Proof of Theorem \ref{thm:disconnected}]
	Let $T_1,\dots,T_\ell$ be the connected components of $H$. Suppose, without loss of generality, that $T_1$ is not a star and has the largest number of edges among the $T_i$'s which are not stars. By permuting the indices, we can also assume that $T_1,\dots,T_k$ are isomorphic to $T_1$, while $T_{k+1},\dots,T_{\ell}$ are not isomorphic to $T_1$. Observe that $T_1$ is not a subgraph of $T_i$ for any $k+1 \leq i \leq \ell$, because either $T_i$ is a star (and hence cannot contain $T_1$), or $e(T_i) \leq e(T_1)$ and $T_i,T_1$ are not isomorphic.
	
	We reduce the problem of computing $\ex(G,kT_1)$ to the problem of computing $\ex(G,H)$; the former problem is NP-hard by Lemma \ref{lem:k_disjoint_copies}. Let $G$ be an input graph with $n$ vertices. Let $G'$ be the graph obtained from $G$ by adding to it, for each $k+1 \leq i \leq \ell$, a collection of $n^2$ disjoint copies of $T_i$ (which are also disjoint from $G$).
	We show that $$\ex(G',H) = e(G') - e(G) + \ex(G,kT_1),$$ which will prove the correctness of the reduction.
	First, observe that if we destroy all copies of $kT_1$ in $G'[V(G)]$ then the resulting subgraph of $G'$ will be $kT_1$-free (as $T_1$ is connected and none of $T_{k+1},\dots,T_{\ell}$ contain $T_1$ as a subgraph). Hence, this subgraph of $G'$ is $H$-free. This shows that $\ex(G',H) \geq e(G') - e(G) + \ex(G,kT_1)$. Note that in particular $\ex(G',H) \geq e(G') - \binom{n}{2}$.
	
	In the other direction, let $F'$ be an $H$-free subgraph of $G'$ with $e(F') = \ex(G',H)$. Since $e(F') \geq e(G') - \binom{n}{2}$, $F'$ must contain (at least) one of the $n^2$ disjoint copies of $T_i$ added to $G'$ for each $k+1 \leq i \leq \ell$. But then $F'[V(G)]$ must be $kT_1$-free, as otherwise $F'$ would contain a copy of $H$. Hence
	$\ex(G',H) = e(F') \leq e(G') - e(G) + \ex(G,kT_1)$, as required. This completes the proof.
\end{proof}

\section{Concluding Remarks and Open Problems}\label{sec:concluding}

As we discussed in Section \ref{sec:intro}, an important special case of Yannakakis's problem which is still open,
asks to characterize the connected graphs $H$ for which computing $\rem_H(G)$ is NP-hard. Combining the result of \cite{ASS} (who proved that
$\rem_H(G)$ is NP-hard for every non-bipartite $H$) and Theorem \ref{thm:tree}, it remains to handle the case of bipartite graphs $H$ which
are not trees.
In our proof of Theorem \ref{thm:tree} in Section \ref{sec:proofmain}, we actually showed that if $T$ is a tree with diameter at most $4$ which is not a star, then it is NP-hard to tell whether an $n$-vertex graph $G$ satisfies $\rem_T(G) = |E(G)|-\ex(n,T)$. Namely, it is NP-hard to decide if $G$ contains an extremal $T$-free graph.
It is thus natural to try
and extend our approach in order to prove that computing $\rem_H(G)$ is NP-hard for every non-star bipartite graph $H$.
This raises the question of the complexity of deciding whether $\rem_H(G) = |E(G)|-\ex(n,H)$ for other bipartite graphs $H$.

Before addressing the case of bipartite $H$, we first observe that when $H$ is the triangle $K_3$, we can in fact decide in polynomial time whether
$\rem_{K_3}(G) = |E(G)|-\ex(n,K_3)$. Indeed, Mantel's theorem \cite{FS} states that $\ex(n,K_3) = \lfloor \frac{n^2}{4} \rfloor$ and
that the only graph meeting this bound is the balanced complete bipartite graph $K_{\lfloor \frac{n}{2} \rfloor, \lceil \frac{n}{2} \rceil}$.
Hence, deciding if $\rem_{K_3}(G) = |E(G)|-\ex(n,K_3)$ is equivalent to deciding if $G$ contains $K_{\lfloor \frac{n}{2} \rfloor, \lceil \frac{n}{2} \rceil}$. To see why this problem is solvable in polynomial time, suppose that $G$'s complement graph has $m$ connected components of sizes $a_1,\ldots,a_m$. It is easy to see that
$G$ contains $K_{\lfloor \frac{n}{2} \rfloor, \lceil \frac{n}{2} \rceil}$ if and only if there is $S \subseteq [m]$ so that $\sum_{i\in S}a_i=\lfloor \frac{n}{2} \rfloor$.
But this latter task can be easily solved in polynomial time using dynamic programming.
A similar argument shows that if $H$ is a non-bipartite edge-critical\footnote{A graph is called {\em edge-critical} if it contains an edge whose deletion decreases the chromatic number.} graph then deciding whether $\rem_H(G) = |E(G)|-\ex(n,H)$ can be done in polynomial time. This follows from the fact that for such graphs $H$, the only $n$-vertex $H$-free graph with $\ex(n,H)$ edges is the Tur\'{a}n graph with $\chi(H) - 1$ parts, see \cite{FS}.

Since the hardness of computing $\rem_H(G)$ is still open only for bipartite $H$, it is more relevant to our investigation here
to determine whether it is NP-hard to tell if $\rem_H(G) = |E(G)|-\ex(n,H)$ for such $H$. Unfortunately, as opposed to the cases when $H$
is a tree or a graph of chromatic number at least 3, we have a very poor understating of $\ex(n,H)$, let alone of the extremal
graphs meeting this bound, see \cite{FS} for more details. The only case which is relatively well-understood is when $H$ is the $4$-cycle
$C_4$. In this case F\"uredi \cite{Furedi} proved that if $q > 13$ is a prime power and $n=q^2+q+1$, then $\ex(n,C_4)=\frac12 q(q+1)^2$.
He further proved (see \cite{FS}) that there is a unique graph meeting this bound (the so called {\em polarity graph}).
It would be very interesting to decide if these facts can be used to show that deciding if
$\rem_{C_4}(G) = |E(G)|-\ex(n,C_4)$ is NP-hard (at least when $n$ is as in F\"uredi's theorem). Again, this is equivalent to
deciding whether an input graph on $n=q^2+q+1$ vertices contains a copy of the polarity graph.

Finally, note that the algorithm described in Lemma \ref{lem:star-forest_low_max_degree} works in time $O(n^{B})$ for a somewhat large $B = B(H)$, i.e., $B = \Theta(v(H)^4)$. It may be interesting to improve this dependence of the exponent on $H$.


\bibliographystyle{plain}

\appendix
\section{Tutte's reduction}
\begin{proof}[Proof of Theorem \ref{thm:f_factor}]
	Let $G$ be a graph and let $f : V(G) \rightarrow \{0,\dots,v(G)\}$. By replacing $f(v)$ with $\min\{d(v),f(v)\}$, we can assume that $f(v) \leq d(v)$. Denote by $m$ the maximum number of edges in a spanning subgraph $F$ of $G$ such that $d_F(v) \leq f(v)$ for every $v \in V(G)$. Construct a graph $G'$ as follows. For each $e = xy \in E(G)$, add two new vertices $e_x, e_y$ and connect them with an edge. Next, for each $x \in V(G)$, add $d(x) - f(x)$ new vertices $x_1,\dots,x_{d(x)-f(x)}$ and connect them to $e_x$ for every edge $e \in E(G)$ with $x \in e$. The resulting graph is $G'$.  Note that the edge-set $\{e_xe_y : e = xy\in E(G)\}$ forms a matching in $G'$. We claim that $\nu(G') = m + \sum_{x \in V(G)}(d(x) - f(x))$, where $\nu(G')$ is the size of a largest matching in $G'$.
	
	First, let $F$ be a spanning subgraph of $G$ with $m$ edges and with $d_F(v) \leq f(v)$ for every $v \in V(G)$. Construct a matching $M$ of $G'$ as follows. For each $e = xy \in E(F)$, add the edge $e_xe_y \in E(G')$ to $M$. Next, for each $x \in V(G)$, let $S(x)$ be the set of edges $e \in E(G) \setminus E(F)$ which touch $x$.  Then $|S(x)| \geq d(x) - f(x)$ because $d_F(x) \leq f(x)$. Also, for each $e \in S(x)$, $e_x$ is not covered by $M$. Take a matching between the sets $\{x_1,\dots,x_{d(x)-f(x)}\}$ and $\{e_x : e \in S(x)\}$ which saturates the former (this is possible as $|S(x)| \geq d(x) - f(x)$), and add this matching to $M$. The resulting matching $M$ has size $m + \sum_{x \in V(G)}(d(x) - f(x))$.
	
	In the other direction, let $M$ be a maximum matching in $G'$. For each $x \in V(G)$ and $1 \leq i \leq d(x) - f(x)$, if $x_i$ is not covered by $M$ then take an arbitrary edge $e =xy \in E(G)$ containing $x$ and replace $M$ with $M - e_xe_y + e_xx_i$. This retains $M$ a matching and does not decrease its size. Hence, we can assume that $x_i$ is covered by $M$ for each $x \in V(G)$ and $1 \leq i \leq d(x) - f(x)$. Let $F$ be the subgraph of $G$ consisting of all edges $e = xy \in E(G)$ such that $e_xe_y \in M$. For each $x \in V(G)$, consider the set $E(x) = \{e_x : e \in E(G), x \in e\}$. In $M$ there are $d(x)-f(x)$ edges which connect a vertex from $E(x)$ with a vertex from $\{x_1,\dots,x_{d(x)-f(x)}\}$. Hence, at most $f(x)$ edges in $M$ connect a vertex in $E(x)$ with a vertex in $E(y)$ for some other $y$. So $d_F(x) \leq f(x)$. This also shows that $|E(F)| = |M| - \sum_{x \in V(G)}(d(x) - f(x))$. This completes the proof.
\end{proof}

\end{document}